\documentclass[11pt,letterpaper]{amsart}
\usepackage[utf8]{inputenc}
\usepackage{amsmath, amsthm, amssymb} 
\usepackage{savesym}
\savesymbol{P}
\usepackage{graphicx, listings}
\usepackage{tikz, wrapfig}
\usetikzlibrary{decorations.markings, arrows.meta}
\usepackage{ytableau, cancel}
\usepackage[dvipsnames]{xcolor}
\graphicspath{ {./images/} }
\usepackage{hyperref}
\hypersetup{
    colorlinks=true,
    linkcolor=black,
    filecolor=black,      
    urlcolor=black,
    citecolor=black
} 

\theoremstyle{definition}
\newtheorem{proposition}{Proposition}
\newtheorem{theorem}{Theorem}

\newtheorem{lemma}{Lemma}
\newtheorem{corollary}{Corollary}

\newtheorem{example}{Example}

\title{Generic Rigidity and the Grassmannian}
\author{Alexander Heaton}
\date{August 21, 2026}

\begin{document}

\begin{abstract}
We characterize the generic rigidity of bar-joint frameworks in all dimensions. The conditions are constructed explicitly from the graph alone, independently of any placement, and are checked by exchange rules on Young tableaux. By following paths in a directed acyclic subgraph of the line graph, we build multisets of red and blue tableaux, called balancing conditions. Then generic rigidity is determined by counting if every semistandard tableau appears the same number of times red as it does blue.
\end{abstract}

\maketitle

\section{Introduction}\label{sec:introduction}

Given a graph $G = (V,E)$ and a placement $p:V \to \mathbb{R}^d$ of its vertices in Euclidean space, we treat the edges as rigid bars, allowed to rotate about joints at each vertex. Such a framework is said to be locally rigid if its only continuous deformations are rigid body motions, and locally flexible otherwise. 
Results of Asimow and Roth \cite{asimow-roth-I, asimow-roth-II} show that for generic frameworks, local rigidity is equivalent to infinitesimal rigidity, and that generic local rigidity is a property of the graph. More precisely, since a framework is a graph and a placement, they show that local rigidity is independent of the coordinates of the vertices, provided they are algebraically independent over $\mathbb{Q}$.
Discovering a characterization of generic rigidity that respects this fact for $d > 2$ has been a long-standing open problem.

In 1864, Maxwell introduced the edge-count $3|V| - 6$ for frameworks in space \cite{maxwell}. \color{black} In modern terminology, a generically minimally rigid framework in $\mathbb{R}^d$ has exactly $d|V| - \binom{d+1}{2}$ edges, and every subgraph on $e'$ edges and $v' \geq d$ vertices satisfies the sparsity condition $e' \leq dv' - \binom{d+1}{2}$. In the plane, these edge-sparsity counts are also sufficient, a fact proved by Pollaczek-Geiringer in 1927 \cite{Pollaczek-Geiringer1} and rediscovered by Laman in 1970 \cite{laman}. For more background, see \cite{connelly1993rigidity, connelly2021frameworks, cruickshank2025rigidity, rigidity-textbook, Izmestiev2019, schulze2017rigidity, whiteley1996matroids}.

However, for $d > 2$, the planar results break down.  The
edge-sparsity counts are no longer sufficient, and the tree-decomposition
formulations equivalent to rigidity in the plane
\cite{crapo, lovasz-yemini} have proven difficult to generalize into higher dimensions.
This is obstructed by the existence of graphs that satisfy the sparsity counts but remain flexible for other reasons, the most famous example being the ``double banana'' in 3-space \cite{hyperbananas, rigidity-textbook}. While rigidity of body-bar frameworks was resolved by Tay \cite{tay} in all dimensions, bar-joint frameworks have resisted similar characterizations. 

White and Whiteley \cite{white-whiteley} characterize the special realizations of a generically isostatic graph by the vanishing of a single bracket polynomial, the pure condition $C(G)$. It is obtained from the determinant of a tied-down rigidity matrix by a series of Laplace expansions on the columns of each vertex in turn, a method they credit to Rosenberg \cite{rosenberg-crm}, followed by division by the tie-down factor $C(T)$. See also \cite[\S 60.4.2]{white-handbook}, where this is described as an algorithm to construct $C(G)$ directly from $G$. 
Their construction of the pure condition takes generic isostaticity as a hypothesis rather than the question. For Maxwell-count subgraphs that are not generically isostatic, they simply set $C(G) = 0$ \cite[\S5]{white-whiteley}.

If the Laplace expansion is combined with Young's straightening law on tableaux, one obtains an explicit and combinatorial way to decide rigidity. To the best of our knowledge this observation has not appeared in the literature, although all of the ingredients are present. For instance, straightening serves as a normal form and zero-test for bracket polynomials in \cite[\S 60.3]{white-handbook} and \cite{sturmfels-whiteley}, in both cases for projective-geometric computations rather than as a criterion for generic rigidity.

Our results differ in several ways. 
First, we construct the balancing conditions directly following paths in the $\Gamma$-oriented line graph $\mathcal{L}$, with no rigidity matrix, no Laplace expansion, and no tie-down. 
Second, we decide whether these balancing conditions vanish using combinatorial exchange rules on Young tableaux.  
Third, we do not assume generic isostaticity: Theorem \ref{thm:main-Maxwell-locally-flexible} decides it. Finally, given an affinely independent placement, we recover explicit formulas for all self-stresses in terms of simplex volumes.

\subsection{Main Results}

Before stating the main results, we explain some notation. By a \textbf{tableau} we mean a blue or red rectangular array of boxes, filled with entries from an ordered set $V$. For any $n$ elements $i_1, \dots, i_n \in V$ let $\pm[i_1 \dots i_n]$ denote the column filled in that order, with commutative product the concatenation of columns, blue for a minus sign, red otherwise, as in
\begin{equation*}
    [147] \equiv \color{red} \begin{ytableau}
        1\\4\\7
    \end{ytableau}\color{black} \equiv \color{blue} \begin{ytableau}
        1\\7\\4
    \end{ytableau} \color{black}
    \quad \text{ or } \quad
    -[254][237] \equiv \color{blue} \begin{ytableau}
        2&2\\5&3\\4&7
    \end{ytableau} \color{black} \equiv \color{blue} \begin{ytableau}
        2&2\\3&5\\7&4
    \end{ytableau} \color{black} \equiv \color{red} \begin{ytableau}
        2&3\\5&2\\4&7
    \end{ytableau}\color{black},
\end{equation*}
taken as alternating in the column entries. Thus $[147]\equiv [471] \equiv -[174]$ and odd permutations of boxes in a column swap the color. A multiset of tableaux like $\{\{ [245][345], [152], [152]\}\}$ may also be denoted by a sum like $[245][345] + [152] + [152]$. Tableaux that strictly increase down columns and weakly increase along rows are called \textbf{semistandard}, and appear frequently in representation theory and geometry.

We further allow tableaux certain \textbf{exchange rules}, following \cite{Fulton}. 
Any tableaux in a multiset may be exchanged as follows. Given a tableau $T$, fix two columns and $r$ boxes in the first column. Replace $T$ by all tableaux obtained by exchanging those $r$ boxes with every choice of $r$ boxes in the second column, maintaining the vertical order, with entries outside those boxes unchanged. Each resulting tableau inherits the same color as $T$, and vanishes if an entry is repeated within a column.

\begin{lemma}\label{lemma:semistandard-tableaux-span}
    Let $n \in \mathbb{Z}_{>0}$. Every multiset of tableaux on a finite, ordered set $V$ with fixed height $n$ may be reduced to a multiset of semistandard tableaux using exchange rules.
\end{lemma}

\begin{proof}
    First, put every column in strictly increasing order by permuting entries and swapping colors if necessary. Then arrange columns into dictionary (also called lexicographic) order. Second, if a tableau $T$ is not semistandard, there is some pair of adjacent columns $C_L,C_R$ and some row $k$ with $C_L(k) > C_R(k)$. Exchange $T$ using the last $n-k+1$ boxes in $C_L$ and all sets $X$ of $n-k+1$ boxes from $C_R$. Every such $X$ has some $x < C_L(k)$ because there are only $n-k$ entries bigger than $C_R(k)$ in $C_R$. If $x$ equals one of the $k-1$ entries above $C_L(k)$ then the resulting tableau vanishes by the alternating property. Otherwise when reordering the new column $C_L'$ to strictly increase, $x$ gets inserted before the $k$th entry, and hence in dictionary order $C_L'$ is strictly smaller than $C_L$. Since the dictionary-smaller columns before $C_L$ are unchanged, and $C_L$ became strictly smaller or vanishes, every tableau $S$ among those replacing $T$ is strictly smaller in dictionary order. Since $V$ is finite, there are finitely many possible fillings, and this process must terminate.
\end{proof}

\textbf{The orientation $\Gamma$}. Let $d \in \mathbb{Z}_{>0}$ and $G=(V,E)$ a finite simple graph. Fix an order on $V$. By a $d$-orientation $\Gamma$ on $G$ we mean an assignment of its edges as sources, streams, sinks, or deserts, as follows. First, for any vertex of degree $\leq d$, assign all its edges as deserts. Repeat this process, assigning edges as deserts if they are incident to any vertex with $d$ or less non-desert edges, until any remaining non-desert edges form a graph with minimum degree $d+1$. For each vertex $a \in V$ in the non-desert graph, choose any $d$ of its adjacent edges as \textit{incoming} at $a$, and the others as \textit{outgoing} at $a$. Later, we will show our results are independent of this choice, which determines $\Gamma$ as follows. 
An edge $(i,j)$ incoming at both endpoints, sometimes denoted $(i,j)_{i,j}$, is a \textbf{source}, because paths in $\mathcal{L}$ will begin at sources. An edge incoming at one endpoint $i$, outgoing at the other $j$, sometimes denoted $(i,j)_i$, is a \textbf{stream}. An edge outgoing at both endpoints, sometimes denoted $(i,j)_\emptyset$, is a \textbf{sink}, because paths in $\mathcal{L}$ will terminate at sinks. The last requirement on $\Gamma$ is to remove cycles. The stream edges form a digraph. Remove any cycles in the stream digraph by replacing a pair of adjacent streams $(a,b)_b$ and $(b,c)_c$ by a sink $(a,b)_\emptyset$ and a source $(b,c)_{b,c}$. This is equivalent to replacing $(a,b)$ by $(b,c)$ in the choice of $d$ edges incoming at $b$.

\textbf{The $\Gamma$-oriented line graph $\mathcal{L}$}. Let $\mathcal{L}$ be the directed graph whose nodes are edges of $G$ and whose arrows point between adjacent edges $\xi \to \eta$ whenever $\xi$ is oriented into a vertex that $\eta$ is oriented out of. Assign the larger endpoint of each source as negative.  Then if $\xi = (a,b)$ is oriented into $a$ and $\eta = (a,c)$ is oriented out of $a$, we say the arrow $\xi \to \eta$ uses $a$, and there are vertices $j_1 <\dots <j_{d-1}$ with edges also oriented into $a$. 
If $\mathfrak{a} \in \mathcal{L}$ leaves a stream or leaves a source from its positive endpoint, associate two weights $\mathfrak{u}_\mathfrak{a} = [j_1 j_2 \dots j_{d-1}ac]$ and $\mathfrak{d}_\mathfrak{a} = [j_1 j_2 \dots j_{d-1} b a]$, called the up-weight and down-weight. If $\mathfrak{a}$ leaves a source from its negative endpoint, use $\mathfrak{u}_\mathfrak{a} = -[j_1 j_2 \dots j_{d-1}ac]$ instead.
Since the stream digraph has no cycles, $\mathcal{L}$ is acyclic. We call the weighted directed acyclic graph $\mathcal{L}$ the $\Gamma$-oriented line graph. The up-weight of a path in $\mathcal{L}$ is the product of the up-weights of its arrows, and similarly for the down-weight.

\textbf{Balancing conditions}. Suppose $\Gamma$ has $\ell$ sinks $\nu_1,\dots, \nu_\ell$. For any $\ell$ sources $\mu_1,\dots,\mu_\ell$, a path system $\mathcal{P}=(\pi,\mathcal{P}_1,\dots,\mathcal{P}_\ell)$ is a permutation $\pi \in S_\ell$ and paths from the sources to the permuted sinks, so that $\mathcal{P}_j$ is a path from $\mu_j$ to $\nu_{\pi_j}$. A path system is disjoint if none of the paths share a node.
The up-weight $u(\mathcal{P})$ of a path system is the product of the up-weights of its paths, and similarly for down-weight $d(\mathcal{P})$. 
The weight $wt(\mathcal{P})$ of a path system is the product of its up-weight with the down-weights of \textit{all other} disjoint path systems connecting the chosen sources to the $\ell$ sinks, times the sign of its permutation. Finally, the \textbf{balancing condition} is the multiset
\begin{equation}\label{eqn:LGV-balancing-condition}
    \{\{ wt(\mathcal{P}) \}\}
\end{equation}
with one element for each disjoint path system between the sources $\mu_1,\dots,\mu_\ell$ and sinks $\nu_1,\dots, \nu_\ell$. Thus there is one balancing condition for each choice of $\ell$ sources. By Lemma \ref{lemma:semistandard-tableaux-span} we can apply exchange rules to each balancing condition until the resulting multiset, denoted $\overline{\{\{ wt(\mathcal{P}) \}\}}$, contains only semistandard tableaux.

Let $\Gamma$ be a $d$-orientation on the graph $G$, independent of any placement. We call $\Gamma$ \textbf{balanced} if $|\text{sinks}|>|\text{sources}|$ or if $|\text{sources}| \geq |\text{sinks}|=\ell$ and for every choice of $\ell$ sources, $\overline{\{\{ wt(\mathcal{P}) \}\}}$ contains each semistandard tableau in equal amounts red and blue. We will prove this condition is well-defined. An all-desert orientation is not balanced. We will also prove that if $\Gamma$ and $\Gamma'$ are two $d$-orientations on $G$, then $\Gamma$ is balanced $\iff$ $\Gamma'$ is balanced. We say $G$ is \textbf{$d$-balanced} if one (and hence every) $d$-orientation $\Gamma$ is balanced.

\begin{theorem}\label{thm:main-Maxwell-locally-flexible}
    Let $G$ have $|V|\geq d$ and $|E| = d|V| - \binom{d+1}{2}$. Then $G$ is $d$-balanced $\iff$ $G$ is generically locally flexible in $\mathbb{R}^d$.
\end{theorem}

\begin{corollary}\label{cor:characterization}
    Let $G$ have $|V| \geq d$ and set $m = d|V| - \binom{d+1}{2}$. Then $G$ is generically rigid in $\mathbb{R}^d$ $\iff$ $G$ contains a spanning $m$-edge subgraph that is not $d$-balanced.
\end{corollary}

The proof of Theorem \ref{thm:main-Maxwell-locally-flexible} uses the Pl\"ucker relations on the Grassmannian of $(d+1)$-dimensional subspaces in $\mathbb{R}^V$. This is particularly fitting in view of Grassmann's original applications of his exterior algebra to statics.
In the Ausdehnungslehre of 1844, he represented forces acting on rigid bodies as line-bound magnitudes of second order, defined the moment of a force as an exterior product, and even described the concept of moment in statics as providing a natural justification for the exterior product and its sign rule \cite{Grassmann1844}.
Also relevant is the projective viewpoint on statics and infinitesimal rigidity developed by Darboux and Sauer \cite{Darboux1993,Sauer1935}. See also \cite{Izmestiev-projective, Izmestiev2019, NixonSchulzeWhiteley2021} for modern treatments.

To our knowledge, Theorem \ref{thm:main-Maxwell-locally-flexible} and Corollary \ref{cor:characterization} give the first characterization of generic rigidity for $d\geq 3$ in which the conditions are constructed directly from the graph and verified by combinatorial operations, without forming or expanding the tied-down rigidity matrix determinant.
Following paths in the $\Gamma$-oriented line graph effectively condenses all information into the balancing conditions. Thus, our characterization uses only the combinatorics of the graph and the combinatorics of Young tableaux. It would be interesting to find a characterization using only the graph, somehow hiding the exchange rules in purely graph-theoretic conditions. The tableaux in $\{\{wt(\mathcal P)\}\}$ are of course intimately related to the graph, arising from disjoint path systems connecting sources to sinks.

Although Theorem \ref{thm:main-Maxwell-locally-flexible} and Corollary~\ref{cor:characterization} give an explicit characterization of generic rigidity, we make no claim that the resulting conditions can be checked in polynomial time. Lemma \ref{lemma:semistandard-tableaux-span} guarantees that the exchange rules terminate, so the criterion is a decision procedure.
But we do not bound the number of exchange rules required, nor the number of tableaux produced along the way. Subsection \ref{subsec:more-efficient-denominator-clearing} will show that the number of columns in the tableaux can always be taken $\leq |E|$.
Understanding the computational complexity of these conditions, and whether their structure can be exploited to obtain an efficient algorithm, remains an interesting open question.

\subsection{Geometric results}
Besides the combinatorial Theorem \ref{thm:main-Maxwell-locally-flexible} above, we also have more general results applying to affinely independent placements. A placement $p:V \to \mathbb{R}^d$ is called \textbf{affinely independent} if any set of at most $d+1$ vertices are affinely independent. In this section, we write the multiset $\{\{ wt(\mathcal{P}) \}\}$ as a sum $\sum wt(\mathcal{P})$. Given an affinely independent placement, we may evaluate $\sum wt(\mathcal{P})$ at $p$ by replacing each column with the corresponding maximal minor of the matrix $M$ whose $i$th column is $(p_i, 1)$, sometimes writing $\sum wt(\mathcal{P})|_p$ for this evaluation at $p$.

When there are non-desert edges, we say $\Gamma$ is \textbf{balanced at $p$} if $|\text{sinks}| > |\text{sources}|$, or if $|\text{sources}| \geq |\text{sinks}|$ and for every choice of $\ell$ sources, the balancing condition $\sum wt(\mathcal{P})$ evaluates to zero at $p$. An all-desert orientation is not balanced at $p$. We will prove that if $\Gamma$ and $\Gamma'$ are two orientations on $G$, then $\Gamma$ is balanced at $p$ $\iff$ $\Gamma'$ is balanced at $p$. Thus, we call $G$ \textbf{balanced at $p$} if one (and hence every) $\Gamma$ is balanced at $p$. 

For $p:V \to \mathbb{R}^d$ we write $p(i)=p_i = (p_{i1},\dots, p_{id}) \in \mathbb{R}^d$.
For each $i,j \in V$ let $e_{ij} = p_j - p_i$, so that $e_{ij} = -e_{ji}$ and for each $(i,j) \in E$ let $w_{ij}=w_{ji}$ be a variable. For each vertex $i \in V$, consider the vertex equation $\sum_j w_{ij} e_{ij} = 0 \in \mathbb{R}^d$, where the sum is over all vertices $j$ with $(i,j) \in E$. Order the edges forming a tuple $w = (w_{ij})$ and let $A$ be the coefficient matrix of the linear system of equations $wA = 0$ corresponding to all vertex equations. When $|V| \geq d$ the framework is infinitesimally rigid if $\dim \text{right ker } A = \binom{d+1}{2}$, and infinitesimally flexible otherwise. 

\begin{theorem}\label{thm:balanced-equivalent-to-selfstress}
    Let $p:V \to \mathbb{R}^d$ be an affinely independent placement of $G$. Then $G$ is balanced at $p$ $\iff$ there exists $w\neq 0$ with $wA = 0$, called a nonzero self-stress.
\end{theorem}

\begin{corollary}\label{cor:affine-Maxwell-infinitesimally-flexible}
    Let $p:V \to \mathbb{R}^d$ be an affinely independent placement of $G$, with $|V| \geq d$ and $|E| = d|V| - \binom{d+1}{2}$. Then $G$ is balanced at $p$ $\iff$ $G,p$ is infinitesimally flexible.
\end{corollary}

Note that Theorem \ref{thm:balanced-equivalent-to-selfstress} applies to graphs with any number of edges. In Section \ref{sec:examples} we give several examples, and in Section \ref{sec:proofs} we give the proofs. In particular, Lemma \ref{lemma:stream-formula-new} of Section \ref{sec:proofs} will give explicit formulas for any self-stress in terms of the sink variables, valid at any affinely independent placement.

\section{Examples}\label{sec:examples}

Before giving examples, we include a simplifying result, convenient for computations, which will be proved in Section \ref{sec:proofs}.

\begin{lemma}\label{lemma:remove-common-factors}
    Let $\{\{ wt(\mathcal{P}) \}\}$ be a balancing condition associated to some choice of sources in a $d$-orientation, such that every tableau contains some common columns. Let $\{\{ wt(\mathcal{P}) \}\}_{\text{reduced}}$ denote the multiset with common columns removed. Then $\overline{\{\{ wt(\mathcal{P}) \}\}}$ contains each semistandard tableau in equal amounts red and blue $\iff$ $\overline{\{\{ wt(\mathcal{P}) \}\}_{\text{reduced}}}$ does, provided all common columns have been permuted into strictly increasing order, and colors of tableaux adjusted accordingly, before removal.
\end{lemma}

Also, to illustrate the $\Gamma$-oriented line graphs $\mathcal{L}$ in Figures \ref{fig:planar-envelope}, \ref{fig:octahedron}, and \ref{fig:planar-envelope-two-sinks}, the weights are drawn as fractions. Let $frac(\mathfrak{a}) = \mathfrak{u}_\mathfrak{a} / \mathfrak{d}_\mathfrak{a}$ and for $\mathcal{P} = (\pi, \mathcal{P}_1,\dots, \mathcal{P}_\ell)$ let $frac(\mathcal{P}) = \text{sgn} \, \pi \cdot \prod_{j=1}^\ell \prod_{\mathfrak{a} \in \mathcal{P}_j} frac(\mathfrak{a})$, the product over all arrows in all paths in the system, times the sign of the permutation.

\begin{example}\label{example:planar-envelope}
For $d=2$, consider the graph $G$ on six vertices with $\Gamma$ given by the sink $(1,2)_\emptyset$, streams $(1,5)_1$, $(2,3)_2$, $(2,4)_2$, $(5,6)_5$ and sources $(1,3)_{1,3}$, $(3,6)_{3,6}$, $(4,5)_{4,5}$, $(4,6)_{4,6}$ with negative endpoints the larger of the two vertices. The $\Gamma$-oriented line graph $\mathcal{L}$ is illustrated in Figure \ref{fig:planar-envelope}, with arrows as circular arcs between edges of $G$ and weights written as fractions $frac(\mathfrak{a}) = \mathfrak{u}_\mathfrak{a}/\mathfrak{d}_\mathfrak{a}$. This makes it easy to visually follow paths from source to sink, and read off their weights.

\begin{figure}[!ht]
\resizebox{0.45\textwidth}{!}{%
\begin{tikzpicture}[
    scale=1.2,
    blue double/.style={
        line width=3.5pt, blue, {Stealth[length=7mm, width=6mm]}-{Stealth[length=7mm, width=6mm]}
    },
    green single/.style={
        line width=3.5pt, green!70!black, -{Stealth[length=7mm, width=6mm]}
    },
    arch/.style={
        very thick, gray!90!black, {Bar[width=3mm, line width=1.5pt]}-{Stealth[length=4mm, width=3mm]}
    },
    frac/.style={
        text=gray!90!black, font=\Large, fill=white, inner sep=2pt
    }
]

\node[font=\Huge\bfseries, inner sep=5pt] (4) at (0, 0) {4};
\node[font=\Huge\bfseries, inner sep=5pt] (5) at (8, 0) {5};
\node[font=\Huge\bfseries, inner sep=5pt] (2) at (0, 12) {2};
\node[font=\Huge\bfseries, inner sep=5pt] (1) at (8, 12) {1};
\node[font=\Huge\bfseries, inner sep=5pt] (3) at (2.5, 8.5) {3};
\node[font=\Huge\bfseries, inner sep=5pt] (6) at (5.5, 3.5) {6};

\draw[ultra thick] (2) -- (1) node[midway, above=2mm, text=red, font=\Huge] {$\boldsymbol{\emptyset}$};

\draw[blue double] (4) -- (5);
\draw[blue double] (6) -- (4);
\draw[blue double] (3) -- (6);
\draw[blue double] (3) -- (1);

\draw[green single] (4) -- (2);
\draw[green single] (5) -- (1);
\draw[green single] (3) -- (2);
\draw[green single] (6) -- (5);

\draw[arch] (2) + (270:1.7) arc (270:360:1.7) node[pos=0.1, anchor=north, frac, xshift=2mm, yshift=-2mm] {$\frac{321}{342}$};
\draw[arch] (2) + (305.5:2.2) arc (305.5:360:2.2) node[midway, anchor=west, frac, xshift=2mm, yshift=-2mm] {$\frac{421}{432}$};

\draw[arch] (1) + (270:1.7) arc (270:180:1.7) node[pos=0.3, anchor=north east, frac, xshift=-2mm, yshift=-2mm] {$\frac{312}{351}$};
\draw[arch] (1) + (212.5:2.2) arc (212.5:180:2.2) node[midway, anchor=north, frac, xshift=-2mm, yshift=-2mm] {$\frac{512}{531}$};

\draw[arch] (3) + (-59:1.1) arc (-59:-234.5:1.1) node[midway, anchor=east, frac, xshift=-3mm] {$\frac{132}{163}$};
\draw[arch] (3) + (32.5:1.5) arc (32.5:125.5:1.5) node[pos=0.1, anchor=south, frac, yshift=2mm] {$\frac{-632}{613}$};

\draw[arch] (6) + (121:1.1) arc (121:-54.5:1.1) node[pos=0.3, anchor=west, frac, xshift=3mm] {$\frac{-465}{436}$};
\draw[arch] (6) + (212.5:1.5) arc (212.5:305.5:1.5) node[pos=0.25, anchor=north, frac, yshift=-2mm] {$\frac{-365}{346}$};

\draw[arch] (4) + (0:1.6) arc (0:90:1.6) node[pos=0.075, anchor=south west, frac, xshift=2mm, yshift=2mm] {$\frac{642}{654}$};
\draw[arch] (4) + (32.5:2.3) arc (32.5:90:2.3) node[midway, anchor=south, frac, xshift=2mm, yshift=2mm] {$\frac{542}{564}$};

\draw[arch] (5) + (180:1.8) arc (180:90:1.8) node[pos=0.2, anchor=east, frac, xshift=-2mm, yshift=2mm] {$\frac{-651}{645}$};
\draw[arch] (5) + (125.5:2.3) arc (125.5:90:2.3) node[midway, anchor=south, frac, xshift=-2mm, yshift=2mm] {$\frac{451}{465}$};

\end{tikzpicture}}
\caption{Both $\Gamma$ and $\mathcal{L}$ are illustrated for Example \ref{example:planar-envelope}, with weights written as fractions and square brackets omitted.}
\label{fig:planar-envelope}
\end{figure}

With only one sink, a path system is just a path, and to check if $G$ is $d$-balanced we need to check $\{\{ wt(\mathcal{P}) \}\}$ for each source. We start with $\mu = (4,6)_{4,6}$. There are two paths from $\mu$ to $\nu = (1,2)_\emptyset$, namely $(4,6)\to(5,6)\to(1,5)\to(1,2)$ and $(4,6)\to(2,4)\to(1,2)$. Reading weights along paths in Figure \ref{fig:planar-envelope} we find $\sum frac(\mathcal{P}) = \frac{542}{564} \frac{321}{342} - \frac{365}{346} \frac{451}{465} \frac{312}{351}$ and hence
\begin{align*}
    \{\{ wt(\mathcal{P}) \}\} &= \{\{ \,\, \color{red} \begin{ytableau}
        3&4&3&5&3\\4&6&5&4&2\\6&5&1&2&1
    \end{ytableau}\color{black}, \,\, \color{blue} \begin{ytableau}
        5&3&3&4&3\\6&4&6&5&1\\4&2&5&1&2
    \end{ytableau} \,\, \color{black} \}\}\\
    \{\{ wt(\mathcal{P}) \}\}_{\text{reduced}} &= \{\{ \,\, \color{blue} \begin{ytableau}
        1&2&3\\3&4&4\\5&5&6
    \end{ytableau}\color{black}, \,\, \color{red}\begin{ytableau}
        1&2&3\\4&3&5\\5&4&6
    \end{ytableau} \color{black} \,\, \}\},
\end{align*}
where in the second line we have permuted boxes to strictly increase down columns, swapped colors accordingly, and removed common columns. The first tableau is already semistandard, since its rows are weakly increasing, so we leave it alone. The second tableau has violations $4>3$ and $5>4$. We apply exchange rules to $4>3$ by swapping everything below $4$, namely $4,5$, with every choice of two boxes in the next column, namely $2,3$ and $2,4$ and $3,4$. Since the first exchange yields a column $[454]$ we only get two new tableaux rather than three, yielding
\begin{align*}
    \{\{ \,\, \color{blue} \begin{ytableau}
        1&2&3\\3&4&4\\5&5&6
    \end{ytableau}\color{black},  \color{red}\begin{ytableau}
        1&4&3\\2&3&5\\4&5&6
    \end{ytableau} \color{black}, \color{red}\begin{ytableau}
        1&2&3\\3&4&5\\4&5&6
    \end{ytableau} \color{black} \,\,  \}\} \\
    =\{\{ \,\, \color{blue} \begin{ytableau}
        1&2&3\\3&4&4\\5&5&6
    \end{ytableau}\color{black},  \color{blue}\begin{ytableau}
        1&3&3\\2&4&5\\4&5&6
    \end{ytableau} \color{black}, \color{red}\begin{ytableau}
        1&2&3\\3&4&5\\4&5&6
    \end{ytableau} \color{black} \,\, \}\},
\end{align*}
where in the second line we have permuted boxes to strictly increase down columns, and swapped colors accordingly. Every tableau is semistandard, so we are done with exchange rules. Additionally, we are done with this example, because already this balancing condition has a semistandard tableau that does not appear the same number of times red as blue. Thus, $G$ is not $2$-balanced, and Theorem \ref{thm:main-Maxwell-locally-flexible} implies it is generically locally rigid in the plane.

We now point out two useful shortcuts. 
\begin{enumerate}
    \item First, if we don't have colored pencils, we can work with $\pm$ rather than red and blue. Also, we can immediately cancel (and stop needlessly rewriting) semistandard tableaux that appear with opposite signs (or colors) while applying exchange rules. In this equivalent setup, every semistandard tableau will appear in equal amounts red and blue exactly when every semistandard tableau cancels. We say the balancing condition \textit{straightens to zero}, since exchange rules are often called \textit{Young's straightening law}. This will be formally justified in Section \ref{sec:proofs}.
    \item Second, in checking $\{\{ wt(\mathcal{P}) \}\}$, we can instead sum over paths the products of ratios of weights along the path, denoted $\sum frac(\mathcal{P})$, which in this example is $\frac{542}{564} \frac{321}{342} - \frac{365}{346} \frac{451}{465} \frac{312}{351}$. We then clear denominators and remove common factors before rewriting as tableaux, where it is easier to apply exchange rules.
\end{enumerate}
\end{example}

\begin{figure}[!ht]
\resizebox{0.55\textwidth}{!}{%
\begin{tikzpicture}[
    scale=1.0, 
    blue double/.style={
        line width=2.5pt, blue, {Stealth[length=3mm, width=6mm]}-{Stealth[length=3mm, width=6mm]}
    },
    green single/.style={
        line width=2.5pt, green!70!black, -{Stealth[length=3mm, width=6mm]}
    },
    arch/.style={
        very thick, gray!90!black, {Bar[width=3mm, line width=1.5pt]}-{Stealth[length=4mm, width=3mm]}
    },
    frac/.style={
        text=gray!90!black, font=\Large, fill=white, inner sep=2pt
    }
]

\node[font=\Huge\bfseries, inner sep=5pt, fill=white] (3) at (0, 0) {3};
\node[font=\Huge\bfseries, inner sep=5pt, fill=white] (4) at (10, 0) {4};
\node[font=\Huge\bfseries, inner sep=5pt, fill=white] (6) at (5, 2) {6};
\node[font=\Huge\bfseries, inner sep=5pt, fill=white] (2) at (3.5, 5) {2};
\node[font=\Huge\bfseries, inner sep=5pt, fill=white] (5) at (6.5, 5) {5};
\node[font=\Huge\bfseries, inner sep=5pt, fill=white] (1) at (5, 11) {1};

\draw[thick, gray!90!black] (2) -- (1) node[midway, left=0.5mm, text=red, font=\large] {$\boldsymbol{\emptyset}$};

\draw[blue double] (3) -- (1);
\draw[blue double] (2) -- (5);
\draw[blue double] (3) -- (4);
\draw[blue double] (6) -- (3);
\draw[blue double] (4) -- (5);
\draw[blue double] (4) -- (6);
\draw[blue double] (5) -- (6);

\draw[green single] (4) -- (1);
\draw[green single] (5) -- (1);
\draw[green single] (3) -- (2);
\draw[green single] (6) -- (2);

\draw[arch] (1) + (284:2.5) arc (284:256:2.5) node[midway, anchor=north, frac, yshift=-2mm] {$\frac{3412}{3451}$};

\draw[arch] (5) + (243:0.9) arc (243:104:0.9) node[midway, anchor=south east, frac, xshift=-2mm, yshift=2mm] {$\frac{2451}{2465}$};

\draw[arch] (6) + (63:1.3) arc (63:117:1.3) node[midway, anchor=south, frac, yshift=2mm] {$-\frac{3462}{3456}$};

\draw[arch] (2) + (296:1.8) arc (296:76:1.8) node[midway, anchor=east, frac, xshift=-2mm] {$\frac{3521}{3562}$};

\end{tikzpicture}}
\caption{Only weights for paths from the source $(5,6)$ are illustrated for Example \ref{example:octahedron}.}
\label{fig:octahedron}
\end{figure}

\begin{example}\label{example:octahedron}
We illustrate the shortcuts above. Let $G$ be the octahedron graph. For $d=3$, let $\Gamma$ have one sink $(1,2)_\emptyset$, seven sources $(1,3)_{1,3}$, $(2,5)_{2,5}$, $(3,4)_{3,4}$, $(3,6)_{3,6}$, $(4,5)_{4,5}$, $(4,6)_{4,6}$, $(5,6)_{5,6}$, and four streams $(1,4)_1$, $(1,5)_1$, $(2,3)_2$, $(2,6)_2$. Consider the source $\mu = (5,6)_{5,6}$ with $6$ as the negative endpoint. We read Figure \ref{fig:octahedron} and find $\sum frac(\mathcal{P}) = \frac{2451}{2465} \frac{3412}{3451} - \frac{3462}{3456} \frac{3521}{3562}$. Clearing denominators and writing as tableaux gives
\begin{equation*}
    \begin{ytableau}
        2&3&3&3 \\ 4&4&4&5 \\ 5&1&5&6 \\ 1&2&6&2
    \end{ytableau} - 
    \begin{ytableau}
        3&3&2&3 \\ 4&5&4&4 \\ 6&2&6&5 \\ 2&1&5&1
    \end{ytableau}.
\end{equation*}
Adjusting signs to make columns strictly increase and then reordering columns into dictionary order we find
\begin{equation*}
    \begin{ytableau}
        1&1&2&3 \\ 2&2&3&4 \\ 3&4&5&5 \\ 4&5&6&6
    \end{ytableau} - 
    \begin{ytableau}
        1&1&2&2 \\ 2&3&3&4 \\ 3&4&4&5 \\ 5&5&6&6
    \end{ytableau}.
\end{equation*}
Every tableau is semistandard, $G$ is not $3$-balanced, and we conclude the octahedron is generically locally rigid in $\mathbb{R}^3$ by Theorem \ref{thm:main-Maxwell-locally-flexible}. 
\end{example}

\begin{example}\label{example:double-banana}
The exchange rules used in Young's straightening law are consequences of the Pl\"ucker relations, which we demonstrate using this example. Theorem 8.45 of \cite{Vinberg} records them as follows. For any $i_1,\dots,i_{d+2}$ and any $j_1, \dots, j_d$ we have
\begin{equation*}
    \sum_{k=1}^{d+2} (-1)^k
    [i_1 \dots \widehat{i_k} \dots i_{d+2}]
    \cdot
    [i_k \, j_1 \dots j_d] = 0.
\end{equation*}

\begin{figure}[!ht]
\resizebox{0.5\textwidth}{!}{%
\begin{tikzpicture}[
  scale=1.0,
  every node/.style={inner sep=3pt, font=\large},
  >={Straight Barb[scale=1.3]},
  sarrow/.style={
    decoration={markings, mark=at position 0.55 with {\arrow[blue, line width=1pt]{>}}},
    postaction={decorate}
  },
  darrow/.style={
    decoration={markings, 
      mark=at position 0.3 with {\arrow[blue, line width=1pt]{<}},
      mark=at position 0.7 with {\arrow[blue, line width=1pt]{>}}
    },
    postaction={decorate}
  }
]

  \node (v5) at (0, 3.5) {5};
  \node (v4) at (0, -3.5) {4};

  \node (v3) at (-2.8, 1.2) {3};
  \node (v1) at (-1.0, 0) {1};
  \node (v2) at (-2.8, -1.2) {2};

  \node (v6) at (2.8, 1.2) {6};
  \node (v7) at (1.0, 0) {7};
  \node (v8) at (2.8, -1.2) {8};

  \draw[sarrow] (v5) -- (v3);
  \draw[sarrow] (v5) -- (v1);
  \draw[sarrow] (v5) -- (v2);

  \draw[darrow] (v3) -- (v1);
  \draw[sarrow] (v3) -- (v2);
  
  \draw[sarrow] (v4) -- (v3);
  \draw[sarrow] (v4) -- (v1);
  \draw[sarrow] (v4) -- (v2);

  \draw[line width=2.5pt, magenta!80!pink] (v1) -- (v2) 
    node[midway, below right, text=magenta!80!pink, inner sep=1pt] {$\emptyset$};

  \draw[darrow] (v5) -- (v6);
  \draw[darrow] (v5) -- (v7);
  \draw[sarrow] (v8) -- (v5);

  \draw[darrow] (v6) -- (v7);
  \draw[sarrow] (v6) -- (v8);
  \draw[darrow] (v7) -- (v8);

  \draw[darrow] (v6) -- (v4);
  \draw[sarrow] (v7) -- (v4);
  \draw[darrow] (v8) -- (v4);

\end{tikzpicture}}
\caption{This illustrates $\Gamma$ for Example \ref{example:double-banana}. Weights are omitted.}
\label{fig:double-banana}
\end{figure}

Let $d=3$ and take $G$ as the double-banana, with $\Gamma$ pictured in Figure \ref{fig:double-banana}, without weights. Consider the source $\mu = (4,8)_{4,8}$ with $8$ as the negative endpoint. There are three paths from $\mu$ to $\nu=(1,2)_\emptyset$ using vertex $4$ and three paths using vertex $8$, yielding 
\begin{multline*}
    \frac{6741}{6784}\frac{3512}{3541} + \frac{6743}{6784}\frac{1532}{1543}\frac{4521}{4532} + \frac{6742}{6784}\frac{3521}{3542} \\
    - \frac{6785}{6748}\left(\frac{6751}{6785}\frac{3412}{3451} + \frac{6752}{6785}\frac{3421}{3452} + \frac{6753}{6785}\frac{1432}{1453}\frac{4521}{4532}\right).
\end{multline*}
After clearing denominators and using the Pl\"ucker relations 
\begin{align*}
    [1467][2345] &= [1345][2467] - [1245][3467] + [1234][4567] \hspace{0.2cm} \text{ and }\\
    [1567][2345] &= [1345][2567] - [1245][3567] + [1235][4567],
\end{align*}
everything cancels. A direct calculation, omitted here, shows the remaining six balancing conditions similarly straighten to zero, implying $G$ is $3$-balanced, and hence generically flexible in $\mathbb{R}^3$.

It is interesting to consider the rigidity matrix with entries $0$ or $\pm (p_{ik} - p_{jk})$ for this example. In a tie-down, we delete the three columns for vertex $1$, two columns for vertex $2$, and one column for vertex $3$, and compute the $18 \times 18$ determinant as a polynomial in the $p_{ik}$. Expanding, $31,104$ of the $18!$ terms are nonzero, but each is a product of $18$ binomials, and we must expand those, yielding
\begin{equation*}
    31,104 \cdot 2^{18} = 8,153,726,976 \quad (\text{over 8 billion!)}
\end{equation*}
signed monomials of degree $18$ in the $p_{ik}$, every one of which is cancelled by others, giving zero in the end. Nothing in the expansion indicates \textit{why} it vanishes. Other tie-downs are worse. In comparison, our balancing condition above vanishes after two Pl\"ucker relations, which also have geometric meaning, and completely avoids any symbolic expansion in the $p_{ik}$.

\end{example}

\begin{example}\label{example:planar-envelope-two-sinks}
We now illustrate path systems with more than one sink, and graphs with $|E| > d|V| - \binom{d+1}{2}$.
The graph in Example \ref{example:planar-envelope} was not $2$-balanced, but if we add the edge $(1,6)$ it becomes $2$-balanced, detecting the self-stress by Theorem \ref{thm:balanced-equivalent-to-selfstress}. Choose $\Gamma$ illustrated in Figure \ref{fig:planar-envelope-two-sinks}, now with two sinks $\nu_1=(1,2)$ and $\nu_2 = (1,6)$.

\begin{figure}[!ht]
\resizebox{0.45\textwidth}{!}{%
\begin{tikzpicture}[
    scale=1.2,
    blue double/.style={
        line width=3.5pt, blue, {Stealth[length=7mm, width=6mm]}-{Stealth[length=7mm, width=6mm]}
    },
    green single/.style={
        line width=3.5pt, green!70!black, -{Stealth[length=7mm, width=6mm]}
    },
    arch/.style={
        very thick, gray!90!black, {Bar[width=3mm, line width=1.5pt]}-{Stealth[length=4mm, width=3mm]}
    },
    frac/.style={
        text=gray!90!black, font=\Large, fill=white, inner sep=2pt
    }
]

\node[font=\Huge\bfseries, inner sep=5pt, fill=white] (4) at (0, 0) {4};
\node[font=\Huge\bfseries, inner sep=5pt, fill=white] (5) at (8, 0) {5};
\node[font=\Huge\bfseries, inner sep=5pt, fill=white] (2) at (0, 12) {2};
\node[font=\Huge\bfseries, inner sep=5pt, fill=white] (1) at (8, 12) {1};
\node[font=\Huge\bfseries, inner sep=5pt, fill=white] (3) at (2.5, 8.5) {3};
\node[font=\Huge\bfseries, inner sep=5pt, fill=white] (6) at (5.5, 3.5) {6};

\draw[ultra thick] (2) -- (1) node[midway, above=2mm, text=red, font=\LARGE] {$\boldsymbol{\emptyset}$};
\draw[ultra thick] (6) -- (1) node[midway, right=2mm, text=red, font=\LARGE] {$\boldsymbol{\emptyset}$};

\draw[blue double] (4) -- (5);
\draw[blue double] (6) -- (4);
\draw[blue double] (3) -- (6);
\draw[blue double] (3) -- (1);

\draw[green single] (4) -- (2);
\draw[green single] (5) -- (1);
\draw[green single] (3) -- (2);
\draw[green single] (6) -- (5);

\draw[arch] (2) + (270:1.7) arc (270:360:1.7) node[pos=0.1, anchor=north, frac, xshift=2mm, yshift=-2mm] {$\frac{321}{342}$};
\draw[arch] (2) + (305.5:2.2) arc (305.5:360:2.2) node[midway, anchor=west, frac, xshift=2mm, yshift=-2mm] {$\frac{421}{432}$};

\draw[arch] (4) + (0:1.6) arc (0:90:1.6) node[pos=0.075, anchor=south west, frac, xshift=2mm, yshift=2mm] {$\frac{642}{654}$};
\draw[arch] (4) + (32.5:2.3) arc (32.5:90:2.3) node[midway, anchor=south, frac, xshift=2mm, yshift=2mm] {$\frac{542}{564}$};

\draw[arch] (5) + (180:1.8) arc (180:90:1.8) node[pos=0.2, anchor=east, frac, xshift=-2mm, yshift=2mm] {$-\frac{651}{645}$};
\draw[arch] (5) + (125.5:2.3) arc (125.5:90:2.3) node[midway, anchor=south, frac, xshift=-2mm, yshift=2mm] {$\frac{451}{465}$};

\draw[arch] (3) + (-59:1.1) arc (-59:-234.5:1.1) node[midway, anchor=east, frac, xshift=-3mm] {$\frac{132}{163}$};
\draw[arch] (3) + (32.5:1.5) arc (32.5:125.5:1.5) node[pos=0.1, anchor=south, frac, yshift=2mm] {$-\frac{632}{613}$};

\draw[arch] (1) + (270:3.2) arc (270:180:3.2) node[pos=0.5, anchor=north, frac] {$\frac{312}{351}$};
\draw[arch] (1) + (270:2.4) arc (270:253.6:2.4) node[pos=0.1, anchor=north, frac] {$\frac{316}{351}$};
\draw[arch] (1) + (212.5:1.8) arc (212.5:253.6:1.8) node[midway, anchor=north east, frac] {$\frac{516}{531}$};
\draw[arch] (1) + (212.5:2.4) arc (212.5:180:2.4) node[midway, anchor=west, frac] {$\frac{512}{531}$};

\draw[arch] (6) + (121:1.4) arc (121:73.6:1.4) node[midway, anchor=south, frac, yshift=2mm] {$-\frac{461}{436}$};
\draw[arch] (6) + (121:2.2) arc (121:-54.5:2.2) node[pos=0.4, anchor=west, frac, xshift=2mm] {$-\frac{465}{436}$};
\draw[arch] (6) + (-147.5:1.4) arc (-147.5:-54.5:1.4) node[midway, anchor=north, frac, yshift=-2mm] {$-\frac{365}{346}$};
\draw[arch] (6) + (-147.5:0.8) arc (-147.5:73.6:0.8) node[pos=0.7, anchor=west, frac, xshift=1mm] {$-\frac{361}{346}$};

\end{tikzpicture}}
\caption{All weights are included as fractions, illustrating both $\Gamma$ and $\mathcal{L}$ for Example \ref{example:planar-envelope-two-sinks}.}
\label{fig:planar-envelope-two-sinks}
\end{figure}

We illustrate the balancing condition arising from choosing sources $\mu_1 = (1,3)$ and $\mu_2 = (3,6)$. The four disjoint path systems connecting $\mu_1 \to \nu_1, \mu_2 \to \nu_2$ get a positive sign while the two disjoint path systems connecting $\mu_1 \to \nu_2, \mu_2 \to \nu_1$ get a negative sign. Paths from $\mu_1 \to \nu_1$ have weights $a_1=\frac{512}{531}$ and $a_2 = -\frac{632}{613}\frac{421}{432}$. Paths from $\mu_2 \to \nu_2$ have weights $\alpha_1 = -\frac{461}{436}$ and $\alpha_2 = -\frac{465}{436}\frac{451}{465}\frac{316}{351}$. The path from $\mu_1 \to \nu_2$ has weight $b_1 = \frac{516}{531}$. Paths from $\mu_2 \to \nu_1$ have weights $\beta_1 = \frac{132}{163}\frac{421}{432}$ and $\beta_2 = -\frac{465}{436}\frac{451}{465}\frac{312}{351}$. Then
\begin{align*}
    \sum frac(\mathcal{P}) &= a_1 \alpha_1 + a_1 \alpha_2 + a_2 \alpha_1 + a_2 \alpha_2 - b_1 \beta_1 - b_1 \beta_2 \\
    &= (a_1 + a_2)(\alpha_1 + \alpha_2) - b_1(\beta_1 + \beta_2).
\end{align*}
After clearing denominators and applying exchange rules, all semistandard tableaux cancel. We omit the computation.

\end{example}

\subsection{A more efficient denominator clearing}\label{subsec:more-efficient-denominator-clearing}

The definition of $wt(\mathcal P)$ in Section~1 was chosen for simplicity. For computation, however, the same balancing condition can
be represented using a substantially smaller common denominator. 

Fix one balancing condition, determined by a choice
\begin{equation*}
    S=\{ \mu_1, \dots, \mu_\ell \}
\end{equation*}
of $\ell$ sources, and let $\mathcal{D}$ denote the corresponding set of disjoint path systems. Note that all arrows leaving a stream node $\xi$ have identical down-weights, which we denote by $d_\xi$, and a disjoint path system contains each stream node at most once.  A source $\mu=(a,b)_{a,b}$ has two possible down-weights, denoted $d_{\mu,a}$ and $d_{\mu,b}$, and each path system uses exactly one of them. Define
\begin{equation*}
    d_S^{\text{eff}} = \left(\prod_{\xi \, \in \, \text{streams}} d_\xi\right) \left( \prod_{\mu=(a,b)\in S} d_{\mu,a}d_{\mu,b} \right).
\end{equation*}
Then $d_S^{\text{eff}}$ is a common denominator for
$\operatorname{frac}(\mathcal P)$ over all
$\mathcal P\in \mathcal{D}$ and the polynomial
\begin{equation*}
    B_S^{\text{eff}} = d_S^{\text{eff}} \sum_{\mathcal{P} \in \mathcal{D}} frac(\mathcal{P})
\end{equation*}
may be used in place of the balancing polynomial $B_S = \sum wt(\mathcal{P})$ from Section \ref{sec:introduction} when carrying out a straightening computation on tableaux. 

There is also a useful bound on its size.  Let $k$, $r$, and $\ell$ denote the numbers of sources, streams, and sinks in $\mathcal{L}$.  Whenever a balancing condition must be checked we have $k \geq \ell$, and therefore
\begin{equation*}
    r+2\ell
    \leq k+r+\ell
    \leq |E|.
\end{equation*}
The product $d_S^{\text{eff}}$ has $r+2\ell$ bracket factors. If a path system uses $q$ arrows, then $q$ of the $r+2\ell$ factors in $d_S^{\text{eff}}$ will cancel $q$ denominator factors in $frac(\mathcal{P})$ and thus every tableau occurring in $B_S^{\text{eff}}$ will have exactly $r+2\ell - q + q = r+2\ell$ columns, and hence at most $|E|$ columns, even before removing any additional common factors.

\subsection{Geometric interpretation} All the tableaux computations above were independent of any placement. However, we can evaluate them geometrically to identify special affinely independent placements and their geometric meaning. We omit the $|_p$ notation and illustrate with examples.

\begin{example}\label{example:three-concurrent-lines}
In Example \ref{example:planar-envelope} the balancing condition for source $(4,6)_{4,6}$ was represented before straightening by $[123][456]([135][245][346] - [145][234][356])$. The Pl\"ucker relations $[135][245]=[125][345]+[145][235]$ and $[235][346]-[234][356]=[236][345]$ transform it into
\begin{equation*}
    [123][345][456] \bigg([125][346]+[145][236]\bigg).
\end{equation*}
Since $[123][345][456]$ is nonzero for any affinely independent placement, the vanishing is controlled by $[125][346]+[145][236]$, whose projective geometry we now explain. Interpret the vectors $P_i=(p_i,1)^T$ as homogeneous coordinates for points in the projective plane.  Let $X=[145]P_2-[125]P_4$. Since $X$ is a linear combination of $P_2$ and $P_4$, it lies on the projective line $24$.  On the other hand,
\begin{align*}
    [15X]
    &= [145][152]-[125][154]\\
    &= -[145][125]+[125][145] = 0.
\end{align*}
Hence $X$ also lies on the line $15$, so $X$ represents the intersection of the lines $15$ and $24$. When does $X$ also lie on the line $36$?  We compute
\begin{align*}
    [X36]
    &= [145][236]-[125][436]\\
    &= [145][236]+[125][346].
\end{align*}
Thus the balancing
condition vanishes precisely when the three lines
\begin{equation*}
    15,\qquad 24,\qquad 36
\end{equation*}
are concurrent. Checking other balancing conditions will similarly produce the same essential geometry, matching the pure conditions of \cite{white-whiteley}.
\end{example}

\begin{example}\label{example:octahedron-quadric}
Recall the octahedron graph in Example \ref{example:octahedron} with 
\begin{equation*}
    \sum frac(\mathcal{P}) = \frac{2451}{2465} \frac{3412}{3451} - \frac{3462}{3456} \frac{3521}{3562}.
\end{equation*} 
For an affinely independent placement, we can find $C \in GL_4$ such that $C[P_2 P_3 P_4 P_5]$ is the identity matrix and set $X=CP_1$, $Y=C P_6$. Applying $C$ to the matrix $M$ we find
\begin{equation*}
    CM = \begin{bmatrix}
        x_1 & 1 & 0 & 0 & 0 & y_1 \\
        x_2 & 0 & 1 & 0 & 0 & y_2 \\
        x_3 & 0 & 0 & 1 & 0 & y_3 \\
        x_4 & 0 & 0 & 0 & 1 & y_4
    \end{bmatrix},
\end{equation*}
The condition $\sum wt(\mathcal{P})|_p = 0$ is therefore equivalent to the condition $\frac{x_2 x_4}{x_1 x_3} = \frac{y_2 y_4}{y_1 y_3}$. We view the constant $k = \frac{y_2 y_4}{y_1 y_3}$ as determined by the position $Y$ of vertex $6$. Therefore, the balancing condition vanishes exactly when vertex $1$ lies on the quadric $x_2 x_4 - kx_1 x_3 = 0$, which we recognize from Blaschke and Liebmann's characterization of infinitesimally flexible octahedra \cite{Blaschke1920, Liebmann1920}.

\end{example}

\begin{example}\label{example:longer-straightening}
Although we have already concluded that $G$ from Example \ref{example:planar-envelope} is generically rigid, to illustrate a longer sequence of exchange rules, we examine the balancing condition for the source $(1,3)_{1,3}$.
Following paths in Figure \ref{fig:planar-envelope}, we find $\sum frac(\mathcal{P}) = \frac{512}{531} - \frac{632}{613} \frac{421}{432}$.
Clearing denominators and reordering brackets gives $-[125][136][234]+[124][135][236]$ which we now rewrite as tableaux. Writing each bracket as a column, we have
\begin{equation*}
    \begin{ytableau}
        1&1&2 \\ 2&3&3 \\ 4&5&6
    \end{ytableau} - 
    \begin{ytableau}
        1&1&2 \\ 2&3&3 \\ 5&6&4
    \end{ytableau}.
\end{equation*}
The first tableau is already semistandard, since the rows are weakly increasing, and hence will remain untouched by exchange rules. The second tableau is not, since $6 > 4$. We apply the exchange rules, and replace the columns $[136][234]$ by the sum of all three ways of exchanging $6$ with an entry of the second column. However, switching $6$ with $3$ is unnecessary since $[133]$ has a repeated index. Thus exchange rules produce
\begin{equation*}
    \begin{ytableau}
        1&1&2 \\ 2&3&3 \\ 4&5&6
    \end{ytableau} - 
    \begin{ytableau}
        1&1&6 \\ 2&3&3 \\ 5&2&4
    \end{ytableau} - 
    \begin{ytableau}
        1&1&2 \\ 2&3&3 \\ 5&4&6
    \end{ytableau}.
\end{equation*}
Adjusting signs to make columns strictly increase, we have
\begin{equation*}
    \begin{ytableau}
        1&1&2 \\ 2&3&3 \\ 4&5&6
    \end{ytableau} + 
    \begin{ytableau}
        1&1&3 \\ 2&2&4 \\ 5&3&6
    \end{ytableau} - 
    \begin{ytableau}
        1&1&2 \\ 2&3&3 \\ 5&4&6
    \end{ytableau}.
\end{equation*}
The $5>3$ in the second tableau is not a violation, since we must swap columns 1 and 2 into dictionary order, giving a semistandard tableau which remains untouched by exchange rules from here onwards. The violation $5>4$ in the third gives two terms after exchange rules, since switching $5$ and $1$ causes a repeat index. We find
\begin{equation*}
    \cancel{\begin{ytableau}
        1&1&2 \\ 2&3&3 \\ 4&5&6
    \end{ytableau}} + 
    \begin{ytableau}
        1&1&3 \\ 2&2&4 \\ 3&5&6
    \end{ytableau} + 
    \begin{ytableau}
        1&1&2 \\ 2&4&3 \\ 3&5&6
    \end{ytableau} - 
    \cancel{\begin{ytableau}
        1&1&2 \\ 2&3&3 \\ 4&5&6
    \end{ytableau}}.
\end{equation*}
Now everything is semistandard except the third tableau, with $4>3$, and the first and fourth tableaux cancel. Applying exchange rules to $4>3$ we now swap everything below $4$, namely the $4,5$ with every choice of two boxes in the next column. Thus we swap $4,5$ with $2,3$, then $2,6$, then $3,6$, maintaining vertical order otherwise.
\begin{equation*}
    \begin{ytableau}
        1&1&3 \\ 2&2&4 \\ 3&5&6
    \end{ytableau} + 
    \begin{ytableau}
        1&1&4 \\ 2&2&5 \\ 3&3&6
    \end{ytableau} +
    \begin{ytableau}
        1&1&4 \\ 2&2&3 \\ 3&6&5
    \end{ytableau} +
    \begin{ytableau}
        1&1&2 \\ 2&3&4 \\ 3&6&5
    \end{ytableau}.
\end{equation*}
Adjusting signs to make columns strictly increase, and swapping columns into dictionary order, we find
\begin{equation*}
    \begin{ytableau}
        1&1&3 \\ 2&2&4 \\ 3&5&6
    \end{ytableau} + 
    \begin{ytableau}
        1&1&4 \\ 2&2&5 \\ 3&3&6
    \end{ytableau} -
    \begin{ytableau}
        1&1&3 \\ 2&2&4 \\ 3&6&5
    \end{ytableau} +
    \begin{ytableau}
        1&1&2 \\ 2&3&4 \\ 3&6&5
    \end{ytableau}.
\end{equation*}
Now the third and fourth tableaux have violations $6>5$ both. Notice that on the third tableau, swapping $6$ and $3$ will cancel the second term above, swapping $6$ and $5$ will cancel the first term above, while swapping $6$ and $4$ will remain. Executing this gives
\begin{equation*}
    \begin{ytableau}
        1&1&3 \\ 2&2&5 \\ 3&4&6
    \end{ytableau} +
    \begin{ytableau}
        1&1&2 \\ 2&3&4 \\ 3&6&5
    \end{ytableau},
\end{equation*}
with $6>5$ remaining in the second term. Applying exchange rules gives
\begin{equation*}
    \begin{ytableau}
        1&1&3 \\ 2&2&5 \\ 3&4&6
    \end{ytableau} -
    \begin{ytableau}
        1&1&4 \\ 2&2&5 \\ 3&3&6
    \end{ytableau} -
    \begin{ytableau}
        1&1&2 \\ 2&3&5 \\ 3&4&6
    \end{ytableau} +
    \begin{ytableau}
        1&1&2 \\ 2&3&4 \\ 3&5&6
    \end{ytableau}.
\end{equation*}
All four remaining tableaux are semistandard, hence the balancing condition for source $(1,3)_{1,3}$ does not straighten to zero.
\end{example}

\section{Proofs}\label{sec:proofs}

In this section we prove our main results. Given a placement $p:V \to \mathbb{R}^d$, let $wA = 0$ represent the system of equations coming from all vertex equations $\sum_{j: (i,j) \in E} w_{ij} e_{ij} = 0$ for each $i \in V$, collectively called the \textbf{self-stress equations}. We call a placement \textbf{affinely independent} if any set of at most $d+1$ vertices are affinely independent. 
Note that by construction, every source or stream has at least one outgoing arrow in $\mathcal{L}$.  Since $\mathcal{L}$ is finite and acyclic, every maximal directed path beginning at a source or stream terminates at a sink.  Hence, whenever there are non-desert edges, there is at least one sink. As in Section \ref{sec:examples}, let $frac(\mathfrak{a}) = \mathfrak{u}_\mathfrak{a} / \mathfrak{d}_\mathfrak{a}$.
We now prove the crucial but elementary lemma. 

\begin{lemma}\label{lemma:stream-formula-new}
    Let $\Gamma$ be a $d$-orientation on $G$ and let $p:V \to \mathbb{R}^d$ be affinely independent. If $wA=0$ and $(i,j)$ is a stream, then
    \begin{equation}\label{eqn:stream-formula-new}
        w_{ij} = \sum_{\nu \, \in \, \text{sinks}} \left( \sum_{\substack{\text{paths } \mathcal{P} \\ ij \to \nu}} \,\,\, \prod_{\mathfrak{a} \in \mathcal{P}} frac(\mathfrak{a})|_p \right) w_\nu,
    \end{equation}
    where the inner sum is over all paths from $(i,j)$ to $\nu$ in $\mathcal{L}$. 
\end{lemma}

\begin{proof}
    First note that $w_{ij} = 0$ whenever $(i,j)$ is a desert edge, by affine independence of $p$ applied repeatedly as we label desert edges.
    Next, in the non-desert edge graph, examine an incoming edge $(a,b)_a$ at an arbitrary vertex $a$, with some outgoing edges $(a,c_1), (a,c_2),\dots,(a,c_m)$, each a stream or sink, which we leave unspecified for now. Note that vertex $a$ must have $d$ incoming edges, whose other endpoints are some $j_1,j_2,\dots,j_{d-1}$ and $b$. The vertex equation corresponding to $a$ in $wA=0$ is
    \begin{equation*}
        \sum_{k=1}^{d-1} w_{aj_k} e_{aj_k} + w_{ab}e_{ab} + \sum_{k=1}^m w_{ac_k} e_{ac_k} = 0.
    \end{equation*}
    Move the last term to the right-side, and apply Cramer's rule to solve for $w_{ab}$, which is valid by affine independence, no matter the choice of $d$ incoming edges. We obtain
    \begin{align*}
        w_{ab} &= \frac{  \det(e_{aj_1}, \dots, e_{aj_{d-1}}, \sum_k(-w_{ac_k})e_{ac_k} )  }{   \det(e_{aj_1}, \dots, e_{aj_{d-1}}, e_{ab})   } \\
        &= \sum_{k=1}^m \frac{  \det(e_{aj_1}, \dots, e_{aj_{d-1}}, e_{ac_k} )  }{   \det(e_{aj_1}, \dots, e_{aj_{d-1}}, e_{ab})   } (-w_{ac_k}).
    \end{align*}
    By rewriting each $e_{a x} = p_x - p_a$, appending a column $(p_a, 1)^T$ on the right while adding a zero to every other column in a new $(d+1)$st coordinate, and then adding the column $(p_a, 1)^T$ to the first $d$ columns, we see that the $d \times d$ determinants in the formula for $w_{ab}$ above are identical to the $(d+1) \times (d+1)$ determinants corresponding to tableaux given by
    \begin{equation*}
        w_{ab} = \sum_{k=1}^m \frac{  [j_1 j_2 \dots j_{d-1} c_k a]|_p  }{   [j_1 j_2 \dots j_{d-1} b a]|_p   } (-w_{ac_k}).
    \end{equation*}
    To eliminate the minus sign, we can swap the last two entries of the numerator, yielding the formula
    \begin{equation}\label{eqn:cramers-rule-new}
        w_{ab} = \sum_{k=1}^m \frac{  [j_1 j_2 \dots j_{d-1} a c_k]|_p  }{   [j_1 j_2 \dots j_{d-1} b a]|_p   } \,\, w_{ac_k}.
    \end{equation}
    Since every path in $\mathcal{L}$ reaching $(a,b)_a$ extends to a path using the arrow $(a,b)_a \to (a,c_k)$ for $k=1,\dots,m$, the formula (\ref{eqn:stream-formula-new}) follows upon repeated substitution of variables, using (\ref{eqn:cramers-rule-new}) at every node along paths starting at $(i,j)$ and ending at a sink.
\end{proof}

\begin{lemma}\label{lemma:source-constraint-equations-new}
    Let $\Gamma$ be a $d$-orientation on $G$ and let $p:V \to \mathbb{R}^d$ be affinely independent. If $wA = 0$, then for each source $\mu$ we must have
    \begin{equation}\label{eqn:source-constraint-equations-new}
        \sum_{\nu \, \in \, \text{sinks}} \left( \sum_{\substack{\text{paths } \mathcal{P} \\ \mu \to \nu}} \,\,\, \prod_{\mathfrak{a} \in \mathcal{P}} frac(\mathfrak{a})|_p \right) w_\nu = 0,
    \end{equation}
    where the inner sum is over all paths in $\mathcal{L}$ from $\mu$ to $\nu$.
\end{lemma}

\begin{proof}
    Since $\mu=(i,j)_{i,j}$ is oriented into both $i$ and $j$, then $w_{ij}$ is determined by a sequence of Cramer's rules starting at $i$ and by another sequence of Cramer's rules starting at $j$. Since $wA=0$, we apply the proof of Equation (\ref{eqn:stream-formula-new}) from Lemma \ref{lemma:stream-formula-new}, treating $(i,j)_{ij}$ first as $(i,j)_i$ and then as $(i,j)_j$, at first disregarding the minus sign on the weight from the negative endpoint, and set both formulas equal. Moving the negative endpoint sum to the other side of the equation recovers the minus sign, and we find formula (\ref{eqn:source-constraint-equations-new}).
\end{proof}

Thus Lemma \ref{lemma:stream-formula-new} determines stream variables $w_{ij}$ in terms of sink variables $w_\nu$, while Lemma \ref{lemma:source-constraint-equations-new} gives constraints on the sink variables $w_\nu$ which, when satisfied, allow Lemma \ref{lemma:stream-formula-new} to determine the source variables $w_\mu$ in terms of the sink variables $w_\nu$, by treating them as streams using the positive endpoint.

Given a $d$-orientation $\Gamma$ on $G$ and an affinely independent placement $p:V \to \mathbb{R}^d$, let the \textbf{$p,\Gamma$-equations} denote the system of equations resulting from $w_{ij} = 0$ when $(i,j)$ is a desert edge, combined with Equations (\ref{eqn:stream-formula-new}) and (\ref{eqn:source-constraint-equations-new}) above, applying (\ref{eqn:stream-formula-new}) to streams and also to sources from their positive endpoint. We call two systems of equations equivalent if they have the same solutions.

\begin{lemma}\label{lemma:stress-system-equivalent-Gamma-system}
    Let $\Gamma$ be a $d$-orientation on $G$ and let $p$ be affinely independent. Then the self-stress equations $wA = 0$ are equivalent to the $p,\Gamma$-equations.
\end{lemma}

\begin{proof}
    Suppose $wA = 0$. We need to show that $w$ also satisfies the $p,\Gamma$-equations. First note that $w_{ij} = 0$ whenever $(i,j)$ is a desert edge by affine independence of $p$. Lemmas \ref{lemma:stream-formula-new} and \ref{lemma:source-constraint-equations-new} complete the proof that $w$ also solves the $p,\Gamma$-equations.
    Next suppose $w$ satisfies the $p,\Gamma$-equations. Since Equation (\ref{eqn:source-constraint-equations-new}) is satisfied, we can apply Equation (\ref{eqn:stream-formula-new}) to source variables, using only paths from their positive endpoint, or only paths from their negative endpoint but ignoring the minus sign on the first arrow's weight, and both formulas will agree by Lemma \ref{lemma:source-constraint-equations-new}. Thus we have values $w_{ij}$ for every $(i,j) \in E$ which satisfy all vertex equations, completing the proof.
\end{proof}

\begin{lemma}\label{lemma:balanced-equivalent-nonzero-solution}
    Let $\Gamma$ be a $d$-orientation on $G$ and let $p$ be affinely independent. Then $\Gamma$ is balanced at $p$ $\iff$ the $p,\Gamma$-equations have a nonzero solution.
\end{lemma}

\begin{proof}
    If every edge is a desert, both sides are false. Assume we have non-desert edges, and hence at least one sink. By their form, the $p,\Gamma$-equations have a nonzero solution exactly when the linear homogeneous equations (\ref{eqn:source-constraint-equations-new}) running over all sources have a nonzero solution. But this happens exactly when there are more sinks than sources, or else when all maximal minors of their coefficient matrix vanish. Say there are $\ell$ sinks $\nu_1,\dots,\nu_\ell$. Choose any $\ell$ sources $\mu_1,\dots, \mu_\ell$ and let $Q_{ij}$ denote the sum over all paths in $\mathcal{L}$ from $\mu_i$ to $\nu_j$ of the product of $frac(\mathfrak{a})|_p$ along the path. The corresponding maximal minor is $\det Q_{ij}$. Let $\mathcal{D}$ be the set of disjoint path systems for $\mu_1,\dots,\mu_\ell$. By the Lindstr\"om–Gessel–Viennot lemma \cite{GesselViennot1985, Lindstrom1973}, $\det Q_{ij}$ is equal to $\sum_{\mathcal{P} \in \mathcal{D}} frac(\mathcal{P})|_p$, where $frac(\mathcal{P})|_p$ is the sign of the permutation times the product of $ \frac{\mathfrak{u}_\mathfrak{a}|_p}{ \mathfrak{d}_\mathfrak{a}|_p}$ over all arrows in all paths in the path system $\mathcal{P}$, evaluated at $p$. Let $D$ be the product of all down-weights of all path systems $\mathcal{P} \in \mathcal{D}$. Then
    \begin{equation*}
        \sum_{\mathcal{P} \in \mathcal{D}} wt(\mathcal{P})|_p = D|_p \cdot \sum_{\mathcal{P} \in \mathcal{D}} frac(\mathcal{P})|_p = D|_p \cdot \det Q_{ij}.
    \end{equation*}
    Because $p$ is affinely independent, $D|_p \neq 0$ and thus $\sum wt(\mathcal{P})|_p = 0 \iff \det Q_{ij} = 0$. The result follows.
\end{proof}

\begin{lemma}\label{lemma:Gamma-equivalent-to-Gamma'}
    Say $\Gamma$ and $\Gamma'$ are two orientations on the same graph $G$ with $p$ affinely independent. Then $\Gamma$ is balanced at $p$ $\iff$ $\Gamma'$ is balanced at $p$. Thus a graph being balanced at $p$ is well-defined.
\end{lemma}

\begin{proof}
    This follows from Lemmas \ref{lemma:stress-system-equivalent-Gamma-system} and \ref{lemma:balanced-equivalent-nonzero-solution}.
\end{proof}

\begin{proof}[Proof of Theorem \ref{thm:balanced-equivalent-to-selfstress}]
The result now follows from Lemmas \ref{lemma:stress-system-equivalent-Gamma-system}, \ref{lemma:balanced-equivalent-nonzero-solution}, and \ref{lemma:Gamma-equivalent-to-Gamma'}.
\end{proof}

\begin{proof}[Proof of Corollary \ref{cor:affine-Maxwell-infinitesimally-flexible}]
    It is well-known that due to isometries of Euclidean space the dimension of the right kernel of $A$ is at least $\binom{d+1}{2}$ when $|V| \geq d$, so the rank of $A$ is at most $d|V| - \binom{d+1}{2}$, which is the number of rows. Hence $G$ is infinitesimally flexible exactly when there exists a nonzero self-stress $wA = 0$. Applying Theorem \ref{thm:balanced-equivalent-to-selfstress} yields the result.
\end{proof}

\subsection{From geometric to combinatorial balancing}

We now move toward proving Theorem \ref{thm:main-Maxwell-locally-flexible}. To justify the combinatorics of Young tableaux, we will need some algebra.

Recall that the Pl\"ucker relations generate the ideal of algebraic relations among the maximal minors of a generic $(d+1) \times v$ matrix. By Young's straightening law, the semistandard tableaux form a basis for the resulting quotient ring, sometimes called the \textbf{bracket ring}, tableaux being products of brackets $[i_1 \dots i_{d+1}]$. Combining this with Lemma \ref{lemma:semistandard-tableaux-span}, we see that applying exchange rules to a sum of tableaux will always result in a unique linear combination of semistandard tableaux. If that linear combination is zero, we say it \textit{straightens to zero}. 
For more details, see any of \cite{Fulton, miller-sturmfels-combinatorial-algebra, bernd-algorithms}. We will quote one result in detail, because it explains the exchange rules.

\begin{proposition}[Page 108 of \cite{Fulton}, attributed to Sylvester.]\label{proposition:Sylvester-exchange}
    For any $p \times p$ matrices $M$ and $N$, and $1\leq k \leq p$,
    \begin{equation*}
        \det(M) \cdot \det(N) = \sum \det(M') \cdot \det(N'),
    \end{equation*}
    where the sum is over all pairs $(M',N')$ obtained from $M$ and $N$ by interchanging a fixed set of $k$ columns of $N$ with any $k$ columns of $M$, preserving the order of the columns.
\end{proposition}

\begin{lemma}\label{lemma:red-blue-equiv-straighten-to-zero}
    Let $\Gamma$ be a $d$-orientation on $G$. Then any $\overline{\{\{ wt(\mathcal{P}) \}\}}$ contains each semistandard tableau in equal amounts red and blue $\iff$ $\sum wt(\mathcal{P})$ straightens to zero. Thus, the conditions for $\Gamma$ being balanced are well-defined.
\end{lemma}

\begin{proof}
    The semistandard tableaux form a basis for the bracket ring \cite[Corollary 3.1.9, page 82]{bernd-algorithms}. Proposition \ref{proposition:Sylvester-exchange} shows that exchange rules are valid algebraic relations in the bracket ring. Writing $\overline{\{\{ wt(\mathcal{P}) \}\}}$ as a linear combination of semistandard tableaux, each red corresponds to a plus sign, and each blue to a minus sign. Therefore, every semistandard tableau appears in equal amounts red and blue exactly when its unique coefficient in that basis is zero.
\end{proof}

\begin{proof}[Proof of Lemma \ref{lemma:remove-common-factors}]
    The Pl\"ucker ideal is prime so this follows from Lemma~\ref{lemma:red-blue-equiv-straighten-to-zero}.
\end{proof}

A bracket polynomial straightens to zero exactly when it belongs to the Pl\"ucker ideal.  However, we need to evaluate our tableaux as maximal minors of the matrix
\begin{equation*}
    M = \begin{bmatrix}
        & & & \\
        p_1 & p_2 & \cdots & p_v\\
        & & & \\
        1 & 1 & \cdots & 1
    \end{bmatrix},
\end{equation*}
which is not generic, but rather has a special affine form. In particular, the maximal minors of $M$ satisfy additional linear relations. To justify using the combinatorial exchange rules, we must show these do not affect our results. 

For a vertex $i\in V$, let $\epsilon_i$ denote the corresponding standard basis vector of $\mathbb{Z}^V$.  Give a tableau the vertex multi-degree
\begin{equation*}
    \deg[i_1 i_2 \dots i_{d+1}] = \epsilon_{i_1} + \epsilon_{i_2} + \cdots + \epsilon_{i_{d+1}},
\end{equation*}
and an edge $(a,b)$ the multi-degree $\deg(a,b) = \epsilon_a + \epsilon_b$.

\begin{lemma}\label{lemma:balancing-multihomogeneous-new}
    Fix a $d$-orientation $\Gamma$ with sinks $\nu_1,\dots,\nu_\ell$, and choose sources $\mu_1,\dots,\mu_\ell$. The balancing condition
    \begin{equation*}
        \sum_{\text{disjoint }\mathcal{P}} wt(\mathcal{P})
    \end{equation*}
    is multi-homogeneous in the vertices of $G$.
\end{lemma}

\begin{proof}
    Consider an arrow $\mathfrak{a}: \xi=(a,b) \longrightarrow \eta=(a,c)$ using the vertex $a$. Then $frac(\mathfrak{a})$ has the form
    \begin{equation*}
        frac(\mathfrak{a}) = \pm \frac{[j_1 \cdots j_{d-1} a c]}{[j_1 \cdots j_{d-1} b a]}.
    \end{equation*}
    The sign has no effect on multi-degree, and hence $\deg frac(\mathfrak{a}) = \epsilon_c - \epsilon_b = \deg \eta - \deg \xi$. Thus degrees telescope along any directed path $\mathcal P:\mu\to\nu$, giving $\deg frac(\mathcal{P}) = \deg \nu - \deg \mu$ for paths beginning at $\mu$ and ending at $\nu$. Therefore every path system connecting sources $\mu_1,\dots,\mu_\ell$ to the sinks $\nu_1,\dots,\nu_\ell$ has identical degree
    \begin{equation*}
        \Delta = \sum_{j=1}^\ell \deg \nu_j - \sum_{j=1}^\ell \deg \mu_j,
    \end{equation*}
    independently of the paths and the permutation matching sources to sinks. Let $\mathcal{D}$ denote the set of all disjoint path systems for $\mu_1,\dots,\mu_\ell$, and $u(\mathcal{P}), d(\mathcal{P})$ the products of up and down weights along $\mathcal{P} \in \mathcal{D}$. Then
    \begin{align*}
        \deg wt(\mathcal{P}) &= \deg u(\mathcal{P}) + \sum_{\substack{\mathcal{Q} \in \mathcal{D}\\ \mathcal{Q} \neq \mathcal{P}}} \deg d(\mathcal{Q}) \\
        &= \Delta + \sum_{\mathcal{Q} \in \mathcal{D}} \deg d(\mathcal{Q}).
    \end{align*}
    The first line is by definition and the second by adding and subtracting $\deg d(\mathcal{P})$. Hence the balancing condition is multi-homogeneous, as needed.
\end{proof}

\begin{lemma}\label{lemma:generic-evaluation-straightening-new}
    Let $T$ be a bracket polynomial with rational coefficients that is multi-homogeneous in the vertices of $G$, and let $p:V\to\mathbb R^d$ be generic. Then $T|_p = 0$ $\iff$ $T$ straightens to zero.
\end{lemma}

\begin{proof}
    If $T$ straightens to zero, then $T$ belongs to the Pl\"ucker ideal, and therefore evaluates to zero on the maximal minors of every $(d+1) \times v$ matrix, including $M$. Thus $T|_p=0$.
    Conversely, suppose $T|_p=0$. After replacing every bracket by the corresponding minor of $M$, the expression $T|_p$ is a polynomial with rational coefficients in the coordinates $p_{ik}$. Since these coordinates are algebraically independent over $\mathbb Q$, the equality $T|_p=0$ implies that this polynomial is identically zero. Therefore $T$ evaluates to zero on every matrix with the affine form of $M$. We need to show that $T$ straightens to zero.

    Now let $\widetilde M$ be any $(d+1) \times v$ matrix whose entries in the last row $\lambda_1, \dots, \lambda_v$ are all nonzero. Let $D = \text{diag}(\lambda_1, \dots, \lambda_v)$ and set $M'=\widetilde{M} D^{-1}$. Then $M'$ has last row $(1,\dots,1)$, so $T|_{M'}=0$. Suppose vertex $i$ occurs with multiplicity $m_i$ in every term of $T$.  By multi-homogeneity, scaling the $i$th column by $\lambda_i$ multiplies every term by $\lambda_i^{m_i}$.  Hence
    \begin{equation*}
        T|_{\widetilde{M}} = \left(\prod_{i=1}^v \lambda_i^{m_i} \right) T|_{M'} = 0.
    \end{equation*}
    Thus $T$ vanishes on the dense open set of $(d+1)\times v$ matrices whose last-row entries are nonzero. It therefore vanishes identically on the space of $(d+1)\times v$ matrices, so $T$ belongs to the Pl\"ucker ideal.  By the straightening law, $T$ straightens to zero.
\end{proof}

\begin{lemma}\label{lemma:generic-geometric-balancing}
    Let $p$ be any generic placement.  Then a $d$-orientation
    $\Gamma$ is balanced $\iff$ $\Gamma$ is
    balanced at $p$.
\end{lemma}

\begin{proof}
    Note the all-desert case is false in both definitions, so we may now assume there is at least one sink. A generic placement is affinely independent.  If there are more sinks than sources, both definitions make $\Gamma$ balanced. Otherwise, Lemmas \ref{lemma:red-blue-equiv-straighten-to-zero}, \ref{lemma:balancing-multihomogeneous-new}, and \ref{lemma:generic-evaluation-straightening-new} imply the result.
\end{proof}

\begin{lemma}\label{lemma:generic-orientation-independent}
    If $\Gamma$ and $\Gamma'$ are two $d$-orientations on $G$, then $\Gamma$ is balanced $\iff$ $\Gamma'$ is balanced. Thus a graph being $d$-balanced is well-defined.
\end{lemma}

\begin{proof}
    Choose a generic placement $p$.  By Lemmas \ref{lemma:Gamma-equivalent-to-Gamma'} and \ref{lemma:generic-geometric-balancing}, $\Gamma$ is balanced $\iff$ $\Gamma$ is balanced at $p$ $\iff$ $\Gamma'$ is balanced at $p$ $\iff$ $\Gamma'$ is balanced.
\end{proof}

\begin{proof}[Proof of Theorem \ref{thm:main-Maxwell-locally-flexible}]
    Let $\Gamma$ be a $d$-orientation and $p$ a generic placement, which is therefore also affinely independent. We have
\begin{align*}
    G \text{ is } d\text{-balanced} & \iff G \text{ is balanced at } p \\
    & \iff (G,p) \text{ is infinitesimally flexible} \\
    & \iff G \text{ is generically locally flexible.}
\end{align*}
The first equivalence follows from Lemmas \ref{lemma:Gamma-equivalent-to-Gamma'}, \ref{lemma:generic-geometric-balancing}, and \ref{lemma:generic-orientation-independent}. The second equivalence is Corollary \ref{cor:affine-Maxwell-infinitesimally-flexible}, and the third follows from the generic equivalence of infinitesimal and local rigidity.
\end{proof}

\begin{proof}[Proof of Corollary \ref{cor:characterization}]
    Let $m=d|V|-\binom{d+1}{2}$, $p$ a generic placement, and $A$ the corresponding rigidity matrix. Since $|V|\geq d$, the space of infinitesimal rigid motions has dimension $\binom{d+1}{2}$, so $\text{rank} A \leq m$. 
    
    Say $G$ is generically rigid. Then $A$ contains $m$ independent rows, with $F \subset E$ the corresponding edge set, and we set $H=(V,F)$. Then $H$ is a spanning subgraph of $G$ whose rigidity matrix consists precisely of these $m$ independent rows, and hence has rank $m$. Thus $H$ is generically rigid. By Theorem \ref{thm:main-Maxwell-locally-flexible}, $H$ is not
    $d$-balanced.

    Now suppose $G$ contains a spanning $m$-edge subgraph $H$ that is not $d$-balanced.  By Theorem \ref{thm:main-Maxwell-locally-flexible}, $H$ is generically rigid.  Thus, for a generic placement $p$, its rigidity matrix has $m$ independent rows, which are also rows of $A$, making $\text{rank} A = m$ as well. Thus $G$ is generically rigid.
\end{proof}

\section*{Acknowledgments}

The author used LLMs for TikZ diagrams, writing code to count the $8,153,726,976$ signed monomials in the tied-down symbolic determinant for the double banana, and for proofreading the paper.

\begingroup
\footnotesize
\bibliographystyle{plain}
\bibliography{references}
\endgroup

\end{document}